\documentclass[12pt, a4paper]{article}
\usepackage{titlesec} 
\usepackage{amsfonts}
\usepackage{amsthm}
\usepackage{amssymb}
\usepackage{amsmath}
\usepackage{theoremref}
\usepackage{enumerate}
\usepackage{bbm}
\usepackage{authblk}
\usepackage{mathtools}

\DeclarePairedDelimiter\abs{\lvert}{\rvert}%
\DeclarePairedDelimiter\norm{\lVert}{\rVert}%

\makeatletter
\let\oldabs\abs
\def\abs{\@ifstar{\oldabs}{\oldabs*}}
\let\oldnorm\norm
\def\norm{\@ifstar{\oldnorm}{\oldnorm*}}
\makeatother

\makeatletter
\g@addto@macro\bfseries{\boldmath}
\makeatother

\newcommand{\A}{\mathcal{A}}
\newcommand{\C}{\mathcal{C}}
\newcommand{\M}{\mathcal{M}}
\newcommand{\T}{\mathbb{T}}
\newcommand{\Ca}{\mathcal{C}}
\newcommand{\K}{\mathcal{K}}
\newcommand{\z}{\zeta}
\newcommand{\conj}[1]{\overline{#1}}
\newcommand{\D}{\mathbb{D}}

\newcommand{\cD}{\conj{\mathbb{D}}}

\newcommand{\m}{\textit{m}}
 
\newcommand{\hd}{\textbf{Hol}(\D)}

\renewcommand\Re{\operatorname{Re}}

\newtheorem{thm}{Theorem}[section]

\newtheorem{cor}[thm]{Corollary}
\newtheorem{prop}[thm]{Proposition}
\theoremstyle{definition}

\theoremstyle{definition}

\titleformat{\section}
{\normalfont\scshape\centering}{\thesection}{1em}{}

\titleformat{\subsection}
[runin]{\bfseries\scshape}{\thesubsection}{0.5em}{}

\titleformat{\subsubsection}
[runin]{}{\thesubsubsection}{0.5em}{}

\begin{document}
\title{\textbf{On model spaces and density of functions smooth on the boundary}}
\date{ }
\author{Adem Limani and Bartosz Malman}
\maketitle

\newcommand{\Addresses}{{
		\bigskip
		\footnotesize
		
		Adem Limani, \\ \textsc{Centre for Mathematical Sciences, Lund University, \\
			Lund, Sweden}\\
		\texttt{adem.limani@math.lu.se}
		
		\medskip
		
		Bartosz Malman, \\ \textsc{KTH Royal Institute of Technology, \\
			Stockholm, Sweden}\\
			\texttt{malman@kth.se}
		
	}}

\begin{abstract}
\noindent 
We characterize the model spaces $K_\Theta$ in which functions with smooth boundary extensions are dense. It is shown that such approximations are possible if and only if the singular measure associated to the singular inner factor of $\Theta$ is concentrated on a countable union of Beurling-Carleson sets. In fact, we use a duality argument to show that if there exists a restriction of the associated singular measure which does not assign positive measure to Beurling-Carleson sets, then even larger classes of functions, such as H\"older classes and large collections of analytic Sobolev spaces, fail to be dense. In contrast to earlier results on density of functions with continuous extensions to the boundary in $K_\Theta$ and related spaces, the existence of a smooth approximant is obtained through a constructive method.

 \end{abstract}

\section{Introduction and main results}

A special feature of the spaces of analytic functions which appear in several operator model theories is that it is often difficult to obtain any explicit example of a non-trivial function that is contained in such a space and has boundary values of some degree of regularity, be it continuous, differentiable or smooth. This principle certainly applies to the classical \emph{model spaces} $K_\Theta$, where the identifiable elements of the space are often limited to the reproducing kernel functions. However, being subspaces of the Hardy space, a function in $K_\Theta$ behaves at least somewhat well on the boundary, in the sense that it admits an almost everywhere defined boundary function. Our goal here is to carry out the investigation of how well such a boundary function can behave for a dense subset of the model space, and more precisely we will focus on differentiable boundary functions.

We will work in the open unit disk $\D$ of the complex plane $\mathbb{C}$, and our boundary functions will thus be living on the unit circle $\T$. We assume the readers familiarity with the usual facts regarding the Hardy spaces $H^p$ on $\D$ and the boundary behaviour of functions in these spaces (all of it can, for instance, be found in \cite{garnett}). In this setting, a model space is constructed in the following way. Let $\nu$ be a finite positive singular Borel measure on $\T$ and construct the associated singular inner function 
\begin{equation} \label{innerformula} S_\nu(z) = \exp\Big(-\int_\T \frac{\zeta + z}{\zeta - z} d\nu(\z)\Big). 
\end{equation}
\noindent
Pick also an arbitrary convergent Blaschke product 
\[ B(z) = \prod_{n=1}^\infty \frac{|a_n|}{a_n} \frac{a_n-z}{1-\conj{a_n}z},
\] 
and set $\Theta := BS_\nu$. The space $\Theta H^2 = \{ \Theta f : f \in H^2 \}$ is then a closed subspace of $H^2$, and the model space is defined as its orthogonal complement: 
\[K_\Theta = H^2 \ominus \Theta H^2.
\] The reproducing kernel of the space is given by 
\begin{equation*}\label{reprodk1} k_\Theta(\lambda,z) = \frac{1-\conj{\Theta(\lambda)}\Theta(z)}{1-\conj{\lambda}{z}}. 
\end{equation*} 
It is clear from this expression that the kernel functions inherit the boundary behaviour of the inner function $\Theta$, which can in general be very irregular. 

During his investigation of the boundary behaviour of model space functions on very fine sets, Aleksandrov established in \cite{aleksandrovinv} the density of the intersection $\A \cap K_\Theta$ in $K_\Theta$, where $\A$ denotes the disk algebra, the algebra of analytic functions in $\D$ with continuous extensions to the boundary. It has been later found in \cite{comptesrenduscont} and \cite{jfabackshift} that Aleksandrov's result admits a generalization of to a wider class of spaces, with proofs using non-constructive approaches based on duality.
\noindent
A natural question is if approximation by functions extending continuously to the boundary is close to the best that one can possibly obtain for a general model space. Similarly, one can ask: what conditions need to be met in order to allow approximation by some more regular class of functions?

Let $E \subset \T$ be a closed set of Lebesgue measure zero and consider the complementary open set $\T \setminus E = U$, which can be written as a countable union $U = \cup_k I_k $ of disjoint maximal open circular arcs $I_k$. Let $|I_k|$ denote the length of the arc $I_k$. The set $E$ is a \emph{Beurling-Carleson} set if the following entropy condition is satisfied:
\begin{gather} \label{thincond} \sum_k |I_k|\log \left( \frac{1}{  |I_k|} \right) < \infty.
\end{gather}
\noindent
Beurling-Carleson sets are characterized as the zero sets on $\T$ of analytic functions on $\D$ with very regular extensions to the boundary. Indeed, Carleson proved in \cite{carlesonuniqueness} that for any such set $E$ and any positive integer $n$, there exists an analytic function with zero set precisely $E$ such that all its derivatives up to order $n$ extend continuously to the boundary. We will denote by $\A^{n}$ the class of analytic functions on $\D$ with derivatives of order $n$ extending continuously to the boundary, and  we also define
\[ \A^{\infty} := \bigcap_{n=1}^\infty \A^{n}.
\] 
Novinger in \cite{novinger1971holomorphic} and Taylor and Williams in \cite{taylor1970ideals} extended Carleson's result and showed that Beurling-Carleson sets are also zero sets of functions in $\A^{\infty}$. For $0<p\leq \infty$, we will also need to define the analytic Sobolev spaces $W^{1,p}_{a}$ on the unit disc $\D$. That is, 
\begin{equation} \label{sobolevdef}
W^{1,p}_{a} = \{f\in \hd: f, f' \in L^{p}(\D, dA) \},
\end{equation} 
where $dA$ denotes the normalized area-measure on $\D$. We regard these spaces as closed subspaces of the classical Sobolev spaces on $\D$, equipped with the standard Sobolev metric $\| f\|_{W^{1,p}} := \|f \|_{L^p} +\|f'\|_{L^p}$. 

Dyakonov and Khavinson found in \cite{starinvsmooth} that there exists model spaces $K_\Theta$ with the property that $W^{1,p}_{a} \cap K_\Theta = \{0\}$, for any $p > 1$. According to their result, this happens precisely when $\Theta = S_{\nu}$ and the singular measure $\nu$ has the property that $\nu(E) = 0$ for any Beurling-Carleson set $E$. Conversely, they show that the model space $K_\Theta$ contains a non-zero function in $\A^\infty$ if either $\Theta$ has a Blaschke factor or if $\nu(E) > 0$ for some Beurling-Carleson set $E$. It is easy to see that any finite positive singular Borel measure $\nu$ on $\T$ can be expressed as a sum  \begin{equation} \label{nu-decomp}
\nu= \nu_{\Ca} + \nu_{\K}
\end{equation}
where the two measures are mutually singular, $\nu_\Ca$ is essentially concentrated on Beurling-Carleson sets, in the sense that there exists an increasing sequence of Beurling-Carleson sets $\{E_n\}_{n \geq 1}$ such that $\lim_{n \to \infty} \nu_\Ca(E_n) = \nu_\Ca(\T)$, while $\nu_K(E) = 0$ for any Beurling-Carleson set. Measures which do not assign positive measure to Beurling-Carleson sets appear notably in the work of Korenblum in \cite{korenblum1981cyclic} and Roberts in \cite{roberts1985cyclic}. For these reasons, in the decomposition in \eqref{nu-decomp}, we will refer to the measure $\nu_\Ca$ as the \emph{Beurling-Carleson part}, and the measure $\nu_K$ as the \emph{Korenblum-Roberts part}. The main result of our investigation is the following.

\begin{thm} \thlabel{maintheorem}
Let $\Theta= BS_{\nu}$ be an inner function with Blaschke product $B$ and singular inner function $S_\nu$ with associated singular measure $\nu$ on $\T$. 
Then the following conditions are equivalent:
\begin{itemize}
\item[(i)]  $\A^{\infty} \cap K_\Theta$ is dense in $K_\Theta$,
\item[(ii)] $W^{1,p}_a \cap K_\Theta$ is dense in $K_\Theta$, for any $1<p\leq \infty$,
\item[(iii)] the Korenblum-Roberts part $\nu_{\K}$ of $\nu$ is trivial. That is, $\nu= \nu_\Ca$.
\end{itemize}
\end{thm}

The equivalence of $(i)$ and $(ii)$ means of course that for a wide range of analytic function spaces $\M$, the density of $\M \cap K_\Theta$ in $K_\Theta$ is equivalent to condition $(iii)$ of \thref{maintheorem}. For instance, the analytic H\"older class $\Lambda_{a}^{\alpha}$, consisting of analytic functions on $\D$ with H\"older continuous boundary values on $\T$ of order $\alpha \in (0,1)$, is a subset of $W^{1,p}_{a}$ for all $p$ satisfying $p < 1/(1-\alpha)$. It follows that $\Lambda_{a}^{\alpha} \cap K_\Theta$ is dense in $K_\Theta$ if and only if $(iii)$ holds, for any $\alpha > 0$. 

The method we employ to establish the density results is to a large extent constructive. To each function $f$ in an admissible model space we construct explicitly a smooth approximant, and the approach is based on co-analytic Toeplitz operators and smoothing out the boundary singularities of the function $f$ by multiplication with a suitable Toeplitz operator symbol. More precisely, every model space $K_\Theta$ satisfying property $(iii)$ in \thref{maintheorem} contains an increasing sequence of model spaces $\{K_{\Theta_n} \}$ such that $\cup_{n \geq 1} K_{\Theta_n}$ is dense in $K_\Theta$ and the functions in each $K_{\Theta_n}$ are sufficiently well-behaved at the boundary so that one can find a co-analytic multiplier $\conj{h}$ making $\conj{h}f$ smooth on $\T$, for each $f \in K_{\Theta_n}$. The Riesz projection then maps $\conj{h}f$ back into the model space and preserves smoothness. A properly constructed sequence of such symbols $\{h_n\}_n$ will give us a corresponding sequence of co-analytic Toeplitz operators which smooth out the singularities and converges to the identity (see \thref{smoothseq}). 

It seems that the approach can be generalized to other spaces with similar properties and on which the co-analytic Toeplitz operators are bounded. For instance, one can easily see that our constructive approach carries over to model subspaces of $H^p$-spaces for $p \in (1,\infty)$. Those spaces can be defined as those $f \in H^p$ for which the function $\conj{\Theta}f$ on $\T$ coincides with boundary values of co-analytic function with a zero at the origin. Our approach breaks down for $p \leq 1$ and $p = \infty$, but nevertheless \thref{maintheorem} as stated, holds even for $p \in (0,1]$, as long as the model spaces are appropriately defined, and for $p = \infty$ if we replace the norm topology on $H^\infty$ with the usual weak* topology. For more details on this see \cite{abstractapproach}. In the cases $p \in (0,1] \cup \{\infty\}$ we know of no constructive approach to establish the existence of approximants. 

To our knowledge, the result on density in model spaces of functions with continuous extensions to the boundary $\T$ so far has not seen a constructive proof, even in the case $p=2$. Our approach breaks down as well. The main difference is that the set of boundary singularities of a function in a general model space is much worse than the set of singularities of functions in $K_\Theta$ where $\Theta$ is carried by a sequence of Beurling-Carleson sets, and that the co-analytic Toeplitz operators and related projections are not good at preserving boundary continuity.

Our theorem treats the density of analytic Sobolev spaces $W^{1,p}_a$ for $p > 1$. When $0 <p < 1$ the situation is simple, since then the spaces $W^{1,p}_a$ contain all the bounded analytic functions, and thus $W^{1,p}_a \cap K_\Theta$ is obviously dense in $K_\Theta$, for all inner functions $\Theta$. In the case of $p = 1$, little is known and the problem seems interesting. In the proof of the necessity of condition (iii) 
in \thref{maintheorem}, we use a duality argument and a technique from \cite{starinvsmooth} which employs a deep result by Korenblum and Roberts from \cite{korenblum1981cyclic} and \cite{roberts1985cyclic}. This result asserts that the inner functions associated with singular measures which do not assign positive measure to Beurling-Carleson sets are cyclic vectors for the forward shift operator on the Bergman spaces. An extension of the result to the space $W^{1,1}_a$ seems to be related to the open problem of determining if the condition $\nu(E)=0$ for any Beurling-Carleson set is sufficient for the associated singular inner function $S_\nu$ to be a weak* cyclic vector for the forward shift operator on the Bloch space. Similar remarks regarding connections to this open problem also appeared in \cite{starinvsmooth} and we refer the reader to \cite{brown1991multipliers} and \cite{anderson1991inner} for further details and progress on this question. Notably, weak* cyclicity has been established for a particular example of a singular inner function in \cite{anderson1991inner}, and consequently there exists an inner function $\Theta$, such that $W^{1,1}_{a}\cap K_\Theta = \{0\}$. In particular, this means that there are model spaces which do not contain non-trivial functions in the so-called \emph{Wiener algebra}, which consists of analytic functions with a Taylor series which converges absolutely in the closed disk $\cD$. As a curious fact we note that if we replace \textit{converges absolutely} by \textit{converges uniformly}, then the situation is straightforward: functions with uniformly convergent Taylor series are dense in $K_\Theta$, for any $\Theta$. This fact has been observed in \cite{abstractapproach}. 

This paper is organized as follows. In the preliminary Section \ref{Prel} we have gathered some background results which will be relevant in the later proofs. Section \ref{SecSuff1} contains the main technical argument, which is the construction of a certain sequence of smoothing functions which converge to the identity. The last section is devoted to the proof of our main result.

\section{Preliminaries} \label{Prel}

We gather a few auxiliary results that will be useful in the proof of the main theorem.

\subsection{Riesz and Herglotz transforms}\label{THsec}

Let $P_+: L^2(\T) \to H^2$ denote the \textit{Riesz projection} given by $$P_+f(z) := \int_{\T}\frac{\zeta}{\zeta-z}f(\zeta)dm(\zeta), \qquad z \in \D,$$ and the \emph{Herglotz transform} by 
\begin{equation*} Hf(z):= \int_{\T}\frac{\zeta + z}{\zeta-z}f(\zeta)dm(\zeta), \qquad  z \in \D,
\end{equation*}
where $dm$ denotes the normalized Lebesgue measure on $\T$. The following result is well-known and simple (see, for instance, \cite{garnett}). Below, $C^\infty(\T)$ denotes the class of infinitely differentiable functions on $\T$.

\begin{prop} \thlabel{rieszcont} The Riesz projection $P_+$ and the Herglotz transform $H$ map $C^{\infty}(\T)$ into $\A^{\infty}$.
\end{prop}

\subsection{A decomposition of singular measures} \label{Modsec}
We present the abstract decomposition result for singular measures which was stated in the introduction. The result is readily established and certainly not new. A version of it has been briefly mentioned in \cite{roberts1985cyclic}.

\begin{prop} \thlabel{Decomp} Let $\nu$ be a positive finite singular Borel measure on $\T$. Then $\nu$ decomposes uniquely as a sum of mutually singular measures \begin{equation} \label{Mdecomp} \nu= \nu_\Ca+ \nu_\K,
\end{equation} such that there exists a sequence of Beurling-Carleson sets $\{E_n\}_{n \geq 1}$ satisfying $$\lim_{n \to \infty} \nu_\Ca(E_n) = \nu_\Ca(\T)$$ and $\nu_\K$ vanishes on every Beurling-Carleson set.
\end{prop}

\begin{proof}[Sketch of proof]
Let $\{E_n\}_{n \geq 1}$ be a sequence of Beurling-Carleson sets such that \[\lim_{n \to \infty} \nu(E_n) = \sup\, \{\nu(E): E\subset \T \, \, \, \text{Beurling-Carleson set } \}. \] Let $\nu|E$ and $\nu|E^c$ denote the restrictions of $\nu$ to $E = \cup_{n \geq 1} E_n$ and $E^c$, respectively. Then it is easy to see that we can take $\nu_\C = \nu|E$ and $\nu_\K = \nu|E^c$, and that such a decomposition is unique. 
\end{proof}

\subsection{Functionals on analytic Sobolev spaces} 
Let $1 < q < \infty$ and denote the classical Bergman spaces by $L^{q}_a$, which are the closed subspaces of $L^{q}(\D,dA)$ consisting of analytic functions on $\D$.  Recall also the definition of the analytic Sobolev spaces $W^{1,p}_a$ in \eqref{sobolevdef}. The following is a well-known Cauchy duality result (see, for instance, \cite{peller1982hankel}).

\begin{prop}\thlabel{duall}
Let $1<p < \infty$ and $q=p/(p-1)$. Then for every $g\in L^{q}_a$ and $f \in W^{1,p}_a$ we have that
\begin{equation}\label{Cauchylim} \lim_{r \rightarrow 1-} \int_{\T} f(r\zeta) \overline{g(r\zeta)} dm(\zeta)
\end{equation} exists, and there exists a constant $C_p >0$, only depending on $p$, such that
\begin{equation}\label{FWp} 
\sup_{0<r<1}\,\abs{ \int_{\T} f(r\zeta) \overline{g(r\zeta)} dm(\zeta) } \leq C_p \norm{f}_{W^{1,p}} \norm{g}_{L^q }.  \end{equation}
\end{prop}

Thus the spaces $W^{1,p}_a$ and $L^q_a$ are duals to each other under the Cauchy dual-pairing in \eqref{Cauchylim}.

\subsection{An estimate for evaluations at the boundary}  Let $E \subset \T$ be a closed set of Lebesgue measure zero, $\nu$ be a measure living on $E$, and $\Theta = S_\nu$. It is well-known that each function $f \in K_\Theta$ has an analytic continuation across $\T \setminus E$. In fact, 
evaluations of any derivative of a function in $K_\Theta$ at any $\lambda \in \T \setminus E$ is a bounded linear functional on $K_\Theta$. Moreover, we have the following crude but important estimate for the norm of these evaluations.

\begin{prop} \thlabel{boundaryevaluationestimate}
Given any positive integer $n$, there exists a constant $C_n > 0$ with the following property: if all functions in $K_\Theta$ are analytic in a neighbourhood of a point $\lambda \in \T$, then the following estimate for derivatives holds: 
\[ |f^{(n)}(\lambda)| \leq C_n\|f\|_{2} \Big( 1 + \sum_{k=1}^{n+1} |\Theta^{(k)}(\lambda)| \Big), \quad f \in K_\Theta. 
\]
\end{prop}

\begin{proof}[Sketch of proof] 
The special setting $n=1$ will exhibit all characteristics of the general case. If every $f \in K_\Theta$ is analytic in a neighbourhood of $\lambda \in \T$, then we have

\[
f'(\lambda) = \lim_{ \xi \to \lambda} \int_{\T} f(\zeta) \conj{\frac{k_\Theta (\lambda , \zeta) - k_\Theta (\xi, \zeta) }{\lambda - \xi}} dm(\zeta).
\]
This means that the difference quotient $ \frac{k_\Theta (\lambda , \cdot) - k_\Theta (\xi, \cdot) }{\conj{\lambda - \xi}}$ converges weakly as $\xi \to \lambda$, to some $D(\lambda, \cdot) \in K_\Theta$. 
It follows that the inner product of any function $f \in K_\Theta$ with $D(\lambda, \cdot)$ equals $f'(\lambda)$ and the norm of the evaluation $\lambda \mapsto f'(\lambda)$ equals $\sqrt{D(\lambda, \lambda)}$. Recall that since $z \mapsto \frac{k_\Theta (\lambda , z ) - k_\Theta (\xi, z) }{\conj{\lambda - \xi}}$ converges uniformly on compact subsets of $\D$ to $\frac{\partial}{\partial \conj{\lambda}} k_\Theta(\lambda, z)$, we have by uniqueness of weak limits that $D(\lambda, z) = \frac{\partial}{\partial \conj{\lambda}} k_\Theta(\lambda, z)$. Now a straight-forward argument involving differentiation of the kernel function with respect to $\conj{\lambda}$ gives
\[ 
\frac{\partial}{\partial \conj{\lambda}} k_\Theta(\lambda, z)= \frac{-\conj{\Theta'(\lambda)}\Theta(z)(1-\conj{\lambda}z) + (1-\conj{\Theta(\lambda)} \Theta(z))z}{(1-\conj{\lambda}z)^2}. 
\] 
Consequently, we can write
\begin{gather*}
D(\lambda, \lambda) = \lim_{r \to 1-} D(\lambda, r\lambda) = \lim_{r \to 1-} \frac{-\conj{\Theta'(\lambda)}\Theta(r\lambda)(1-r) + (1-\conj{\Theta(\lambda)} \Theta(r\lambda))r\lambda}{(1-r)^2}. 
\end{gather*}
Employing the classical l'Hopital's rule twice, we readily verify that $D(\lambda,\lambda)$ is a linear combination of bounded terms and products of the form $\conj{\Theta'(\lambda)}\Theta(\lambda)$, $\conj{\Theta(\lambda)}\Theta''(\lambda)$ and alike. Thus
\[ |D(\lambda, \lambda)| \leq C \Big(1 + |\Theta'(\lambda)| + |\Theta''(\lambda)|\Big)^2, \] and the stated estimate follows from this.
\end{proof}

\section{A smoothing approximation to the identity} \label{SecSuff1}

We will make use of a construction, initially due to Carleson in \cite{carlesonuniqueness} and later refined by Taylor and Williams in \cite{taylor1970ideals}. Let $E$ be a Beurling-Carleson set on $\T$, with complementary open set $U$ that decomposes into $U = \cup_k I_k$, where $I_k$ are disjoint maximal circular arcs of $\T$. Without loss of generality, we may assume that none of the arcs include the point $1 \in \T$. Then, if $I_k$ is extending from $\alpha_k = e^{ia_k}$ to $\beta_k = e^{ib_k}$, we will have $0 \leq a_k< b_k < 2\pi$ and because $E$ is a Beurling-Carleson set, by definition we have that 
\[\sum_k (b_k-a_k)\log \left(\frac{1}{b_k - a_k}\right) < \infty.
\] 
Taylor and Williams use this assumption to construct a function $h(t)$ defined for $t \in [0, 2\pi]$, with the properties that $h(t) = \infty$ on the complement of the union of the intervals $(a_k, b_k)$, on those intervals $h$ is of class $C^\infty$ and 
\[\lim_{t \to a_k^+} h(t) = \lim_{t \to b_k^-} h(t) = \infty,
\]
in such a way that $h(t)$ is integrable on $[0,2\pi)$ and the outer function 
\begin{gather} \label{gfunc} g(z) = \exp\Bigg( -\frac{1}{2\pi}\int_0^{2\pi} \frac{e^{it}+z}{e^{it}-z} h(t) dt  \Bigg) 
\end{gather} 
belongs $\A^{\infty}$ with zero set precisely equal to $E$. Moreover, and this is crucial for our further purposes, the construction of Taylor and Williams is so that 
\begin{gather} \label{gzeros}
E \subseteq \bigcap_{n=1}^\infty \{ \zeta \in \T : g^{(n)}(\zeta) = 0\}.
\end{gather} 
Further details of their construction can be found in \cite[Theorem 3.3]{taylor1970ideals}. 

Let $\delta_E(z)$ denote the distance from a point $z \in \cD \setminus E$ to the set $E$. It follows easily from \eqref{gzeros} that \begin{equation} \label{gest} |g^{(k)}(z)| \leq C_n \, \delta_E(z)^n
\end{equation} for any $n \geq 0$, $g^{(k)}$ being the $k$:th derivative. Indeed, if $z \in \D$ and $w \in E$ is a point such that $|z - w| = \delta_E(z)$, then let $L$ be the line segment between $z$ and $w$. We have \[|g(z)| = |g(z) - g(w)| = \Big| \int_L g'(s) ds \Big| \leq \delta_E(z) \cdot \sup_{s \in L} |g'(s)|.\] Applying the same reasoning to $g'(s)$, we obtain $|g(z)| \leq \delta_E(z)^2 \cdot \sup_{s \in L} |g''(s)|$ and we can iterate to obtain \eqref{gest} for $k = 0$. The estimate for derivatives follows from the invariance of \eqref{gzeros} under differentiation.

\begin{prop}
\thlabel{shapirolemma} Let $\nu$ be a positive singular measure supported on a Beurling-Carleson set $E$, $B$ be a finite Blaschke product, and set $\Theta = B S_\nu$. If $g \in \A^\infty$ satisfies \eqref{gzeros}, then for any function $f \in K_{\Theta}$, we have that $f\conj{g}$ is of class $C^\infty(\T)$. 
\end{prop}

\begin{proof}
The function $f\conj{g}$ is already $C^\infty$ on the complement of the set $E$, since both $f$ and $g$ are. Moreover, using the fact that $B$ is analytic across $\T$ we easily establish the following simple estimate \[ |\Theta^{(n)}(z)| \leq C(n) \int_{E} \frac{1}{|1 - \conj{\zeta}z|^{n+1}} d\nu(\zeta) \leq C(n) \delta_E(z)^{-n-1} \nu(\T).\] Together with \thref{boundaryevaluationestimate} and \eqref{gest}, this implies that derivatives of any order of $f(e^{it})\conj{g(e^{it})}$ tend to zero as $e^{it} \to E$. It follows that $f\conj{g} \in C^\infty(\T)$. 
\end{proof}

We will now modify the construction of the function $g$ above to obtain a sequence of analytic functions which smooths out singularities of functions in our spaces and converges to the identity.

\begin{prop} \thlabel{multprop}
Let $\nu$ be a positive measure supported on a Beurling-Carleson set $E$, $B$ a finite Blaschke product, and $\Theta = B S_\nu$. Then there exists a sequence of analytic functions $\{H_n\}_{n\geq 1}$ in $\A^\infty$ with following properties:

\begin{enumerate}[(i)]
\item the boundary function of $f\conj{H_n}$ is of class $C^\infty(\T)$ for any $f \in K_{\Theta}$,
\item $\lim_{n \to \infty} H_n(\zeta) = 1$ holds for Lebesgue almost every $\zeta \in \T$,
\item $\sup_{n} \|H_n\|_\infty < \infty$.
\end{enumerate}
\end{prop}

\begin{proof} Let $\psi_n$ be an increasing sequence of $C^\infty$-functions with $0 \leq \psi_n \leq 1$, compactly supported in the interval $(0,1)$, and such that $\psi_n(t) = 1$ for all $t \in [1/n, 1-1/n]$. Let $U = \cup_{k=1}^\infty (a_k, b_k)$ be as defined in the beginning of the section and consider the functions
\[\phi_n(t) = \sum_{k=1}^n \psi_n \left( (t-a_k) /(b_k-a_k) \right).
\] 
Certainly $\phi_n$ is $C^\infty$ and $\lim_{n \to \infty} \phi_n(t) = 1$ almost everywhere with respect to the Lebesgue measure. Note that $\phi_n$ is identically equal to zero in a neighbourhood of any of the endpoints of the intervals appearing in the decomposition of $U$. Let the function $h$ and corresponding analytic function $g$ given by \eqref{gfunc} be as in the construction by Williams and Taylor described above. It follows from the construction that $h\phi_n$ is a $C^\infty$-function, and that $h(1-\phi_n)$ converges pointwise to zero almost everywhere, which by dominated convergence theorem applied to $|h(1-\phi_n)| \leq |h|$ implies that 
\begin{equation} \label{intzero} \lim_{n \to \infty} \int_0^{2\pi} h(t)(1-\phi_n(t))dt = 0. 
\end{equation}
We set \begin{gather*} H_n(z) = \exp\Bigg( -\frac{1}{2\pi}\int_0^{2\pi} \frac{e^{it}+z}{e^{it}-z} h(t)(1-\phi_n(t)) dt  \Bigg).
\end{gather*}
Note that certainly $H_n$ are uniformly bounded, so $(iii)$ is satisfied. By construction, any point $\zeta= e^{it}$ with $t \in U$ will have an open neighbourhood around it on which $h(1-\phi_n)$ vanishes for sufficiently large $n$. Thus for sufficiently large $n$ the functions $H_n$ will have analytic continuations to a neighbourhood of any fixed $\zeta= e^{it}$ with $t \in U$. By \eqref{intzero} we deduce that $\lim_{n \to \infty} H_n(\zeta) = 1$ for each such $\zeta$. Thus $\lim_{n \to \infty} H_n(\zeta) = 1$ for almost every $\zeta \in \T$, which is the assertion in $(ii)$. 
We can write $H_n = gF_n$, where $F_n$ is the exponential of the Herglotz transform of $h\phi_n$. Since $h\phi_n$ belongs to $C^\infty(\T)$, we have that $F_n$ has a boundary function in $C^\infty(\T)$ by \thref{rieszcont}. According to \thref{shapirolemma}, we now conclude that $f\conj{H_n} = f\conj{gF_n}$ is of class $C^\infty(\T)$ for any $f \in K_{\Theta}$ which proves the assertion in $(i)$.
\end{proof}

We have the following immediate corollary.

\begin{cor} \thlabel{smoothseq}
Let $\nu$ be a positive singular measure supported on a Beurling-Carleson set $E$, $B$ be a finite Blaschke product, and $\Theta = B S_\nu$. Then there exists a sequence of Toeplitz operators $T_n: K_\Theta \to K_\Theta \cap \A^\infty$, which converge to the identity operator in the strong operator topology. 
\end{cor}

\begin{proof}
The operators $T_n$ are given by $T_nf = P_+ \conj{H_n}f$, where $\{H_n\}_{n \geq 1}$ is the sequence constructed in \thref{multprop}. From $(i)$ of that same proposition together with \thref{rieszcont} we get that $T_n f \in \A^\infty$, and since $K_\Theta$ is well-known to be invariant under Toeplitz operators with co-analytic symbols, we conclude that $T_nf \in K_\Theta \cap \A^\infty$ for all $f \in K_\Theta$. According to $(ii)$ and $(iii)$ of \thref{multprop}, we may apply the dominated convergence theorem to deduce that $T_n$ converges to the identity in the strong operator topology. 
\end{proof}

\section{Proof of the main result} \label{PoMR}

\subsection{Proof of the sufficiency}

Note that \thref{smoothseq} essentially establishes a constructive approximation scheme in the special case when the singular measure $\nu$ in the definition of the singular inner part of $\Theta$ is carried by a single Beurling-Carleson set. Our task here is thus limited to carrying out a simple verification that we can extend our approximations to a measure $\nu$ supported on a countable union of such sets.

\begin{proof}[Proof of (iii) $\Rightarrow$ (i) of \thref{maintheorem}] Let $\Theta = BS_\nu$, where $B$ is a Blaschke product, and $S_\nu$ is a singular inner function with corresponding singular measure $\nu$. Suppose that the Korenblum-Roberts part $\nu_{\K}$ vanishes. Then it follows from \thref{Decomp} that there exists a sequence of Beurling-Carleson sets $\{E_n\}_{n}$, such that $\nu(E_n) \to \nu(\T)$ as $n \to \infty$. Let $\{a_n\}_{n\geq 1}$ be the sequence of zeros of $B$,$B_n$ be the finite Blaschke product with zeros at $a_1, \ldots, a_n$ and $S_{\nu_n}$ be the singular inner function defined by the measure $\nu_n$ which is the restriction of $\nu$ to the Beurling-Carleson set $E_n$. For $\Theta_n = B_nS_{\nu_n}$, we have that $K_{\Theta_n} \subset K_\Theta$, because $\Theta_n$ divides $\Theta$. Moreover, we have \begin{gather*} \int_\T |\Theta_n -  \Theta|^2 d\m = \int_\T |1 - (B/B_n) S_{\nu - \nu_n}|^2 dm \\ = 2 - 2 \Re \int_\T (B/B_n)S_{\nu - \nu_n} d\m \\ = 2 - 2 (\prod_{k \geq n+1} |a_k|)\big(\nu(\T) - \nu_n(\T)\big). \end{gather*} The last expression clearly tends to zero as $n \to \infty$, and thus we conclude that $\lim_{n \to \infty} \|\Theta_n - \Theta\|_{2} \to 0$. By passing to a subsequence, we can assume that $\Theta_n$ converges to $\Theta$ pointwise almost everywhere on $\T$. Let $P_{K_\Theta}$ denote the orthogonal projection from $H^2$ onto $K_\Theta$. It is easy to see that this projection is given by \[ P_{K_\Theta}f = f - \Theta P_+ \conj{\Theta} f,\] where $P_+$ is the Riesz projection. Let $P_{K_{\Theta_n}}$ denote the corresponding projection onto $K_{\Theta_n}$. Then, for $f \in K_\Theta$ we have that 
\begin{gather*} P_{K_{\Theta}}f - P_{K_{\Theta_n}}f  = f - P_{K_{\Theta_n}}  = \Theta_n P_+ \conj{\Theta_n} f \\ = \Theta_n P_+ (\conj{\Theta_n - \Theta})f + \Theta_n P_+\conj{\Theta} f \\ = \Theta_n P_+(\conj{\Theta_n - \Theta}) f.
\end{gather*}
Because $|\Theta_n| = 1$ on $\T$, the dominated convergence theorem implies now that the sequence of projections $P_{K_{\Theta_n}}$ converges in the strong operator topology to the identity on $K_\Theta$. 

Let $f \in K_\Theta$ be an arbitrary function. By the above development, for any $\epsilon > 0$ there exists a function $f_\epsilon \in K_\Theta$ which is contained in a subspace $K_{\Theta_\epsilon} \subset K_\Theta$ to which \thref{smoothseq} applies and such that $\|f - f_\epsilon\|_{2} < \epsilon/2$. By \thref{smoothseq}, there exists a sequence of Toeplitz operators $\{T_n\}_{n \geq 1}$ such that $T_n f_\epsilon \in K_{\Theta} \cap \A^\infty$, and such that $\|f_\epsilon - T_n f_\epsilon\|_{2} < \epsilon/2$ for sufficiently large $n$. Thus $\|f - T_nf_\epsilon\|_{2} < \epsilon$, and $\A^\infty \cap K_\Theta$ is dense in $K_\Theta$. 

\end{proof}

\subsection{Proof of the necessity}

Finally, we prove the necessity of the structure of the singular measure to ensure smooth approximations being possible. In contrast to previous section, our approach is non-constructive and we use duality.

\begin{proof}[Proof of (ii) $\Rightarrow$ (iii) of \thref{maintheorem}]

Let $\Theta= B S_\nu $ be an inner function such that (iii) fails, that is, the Korenblum-Roberts part $\nu_\K$ is not identically zero. Now factorize $\Theta= \Theta_\Ca S_\K $, where $S_{\K}= S_{\nu_\K}$ and $\Theta_{\Ca}= B S_{\nu_\Ca}$ and observe that a simple computation involving reproducing kernels establishes the orthogonal decomposition
\[ K_\Theta = K_{\Theta_{\Ca}} \oplus \Theta_{\Ca} K_{S_\K}.
\] Our aim is to prove that no non-zero function belonging to the subspace $\Theta_\Ca K_{S_\K}$ can be approximated by $W^{1,p}_a \cap K_\Theta$, for any $p>1$. To this end, let $p>1$ be arbitrary and fixed. Since $\nu_{\K}(E)=0$ for every Beurling-Carleson set $E\subset \T$, the Korenblum-Roberts Theorem (see Theorem 1 in \cite{korenblum1981cyclic} or Theorem 2 in \cite{roberts1985cyclic}) implies that the singular inner function $S_\K$ is a cyclic vector on the Bergman spaces $L^{q}_{a}$, for any $q>1$. That is, the subspace $S_{\K} L^{q}_{a}$ is dense in $L^{q}_{a}$. Now let $f\in H^\infty \cap K_{S_\K} \subset L^q_a$ and pick a sequence of analytic polynomials $\{p_n \}_{n\geq 1}$ such that $S_{\K} p_n \rightarrow f$ in $L^q_a$, as $n \to \infty$. Since multiplication with the bounded analytic function $\Theta_{\Ca}$ is a continuous operation on $L^q_a$, we have that $\Theta p_n \to \Theta_{\Ca} f$ in $L^q_a$. Observe that the limit function $\Theta_{\Ca}f$ belongs to $K_\Theta$, since $\Theta_\Ca K_{S_{\K}} \subset K_\Theta$. According to \thref{duall}, the Cauchy dual-pairing with an $L^{q}_a$-function induces a bounded linear functional on $W^{1,p}_a$, where $p>1$ satisfies $p=q/(q-1)$. In particular, this implies that for any $g\in W^{1,p}_a \cap K_{\Theta}$ we have the following
\[ \int_{\T} \Theta_{\Ca}(\zeta) f(\zeta) \overline{g(\zeta)} dm(\zeta)= \lim_{n \rightarrow \infty} \int_{\T} \Theta(\zeta) p_n(\zeta) \overline{g(\zeta)} dm(\zeta) = 0.
\]
The last equality follows from the fact that $g$ is orthogonal to $\Theta p_{n}$, for each $n\geq1$. This shows that $\Theta_{\Ca} f \in (W^{1,p}_a \cap K_\Theta )^\perp$ for any bounded function $f \in K_{S_\K}$, hence we actually get the inclusion $\Theta_\Ca K_{S_\K} \subseteq (W^{1,p}_a \cap K_\Theta )^\perp$. This obviously implies that $W^{1,p}_a \cap K_ \Theta$ is not dense in $K_{\Theta}$, for any $p > 1$. 
\end{proof}

\bibliographystyle{siam}
\bibliography{mybib}

\Addresses

\end{document}